\documentclass[12pt,a4paper]{article}   
\usepackage[utf8]{inputenc}
\usepackage[english]{babel}

\usepackage[theorems, global-numbering, colorlinks]{default}

\usepackage[normalem]{ulem}

\usepackage[left=2cm,right=2cm,top=2.5cm,bottom=2cm,headheight=15pt]{geometry}

\usepackage{dsfont}

\newcommand{\e}{\mathrm e}

\newtheorem{InformalTheorem}[theorem]{Informal Theorem}
\newtheorem{InformalClaim}[theorem]{Informal Claim}

\theoremstyle{definition}
\newtheorem{example}[theorem]{Example}

\title{Probabilistic interpretation of the Selberg--Delange~Method in analytic number theory}
\author{Maximilian Janisch\footnote{Institut für Mathematik, Universität Zürich. E-Mail: \url{maximilian.janisch@math.uzh.ch}.}}
\date{January 29, 2025}

\newtheorem{claim}{Claim}

\begin{document}
	\maketitle

    \begin{abstract} 
        In analytic number theory, the Selberg--Delange Method provides an asymptotic formula for the partial sums of a complex function $f$ whose Dirichlet series has the form of a product of a well-behaved analytic function and a complex power of the Riemann zeta function. In probability theory, mod-Poisson convergence is a refinement of convergence in distribution toward a normal distribution. This stronger form of convergence not only implies a Central Limit Theorem but also offers finer control over the distribution of the variables, such as precise estimates for large deviations.
    
        In this paper, we show that results in analytic number theory derived using the Selberg--Delange Method lead to mod-Poisson convergence as $x \to \infty$ for the number of distinct prime factors of a randomly chosen integer between $1$ and $x$, where the integer is distributed according to a broad class of multiplicative functions.

        As a Corollary, we recover a part of a recent result by Elboim and Gorodetsky under different, though related, conditions: A Central Limit Theorem for the number of distinct prime factors of such random integers.
    \end{abstract}
    
	\tableofcontents

    \section{Introduction}
    A classical result in number theory, the Erd\H{o}s-Kac Theorem \cite{ErdosKac}, states that if, for $x\in\N$, $N_x$ is a uniformly distributed random variable on $\set{1,\dots,x}$, then the random variables
    \begin{equation*}
        \frac{\omega(N_x)-\ln \ln x}{\sqrt{\ln \ln x}}
    \end{equation*}
    converge in distribution to a standard normal distribution as $x\to\infty$, where here and henceforth, $\omega(n)$ denotes the number of distinct prime factors of $n\in\N$.

    The notion of \emph{mod-$\phi$ convergence} is a strengthening of convergence in distribution to a normal distribution. It was first introduced in \cite{FMN}, and we give a recall of the Definition of mod-$\phi$ convergence, as well as a special case of it, mod-Poisson convergence, in section \ref{sect:mod-phi}. In \cite[Section 7.2]{FMN}, it was established that the $N_x$ converge mod-Poisson, so that the Erd\H{o}s-Kac Theorem was recovered there as a Corollary.

    In \cite{Elboim-Gorodetsky}, a different but related set of random variables has been studied, based on the notion of \emph{multiplicative functions}.
    \begin{definition}[Multiplicative functions]
        A function $\alpha:\N\to\C$ is called \emph{multiplicative} if, for any co-prime integers $n,m\in\N$, we have
        \begin{equation*}
            \alpha(nm) = \alpha(n) \alpha(m).
        \end{equation*}
    \end{definition}
    Given a multiplicative function $\alpha:\N\to\R_+$ which is not identically equal to $0$, we define, as is done in \cite{Elboim-Gorodetsky}, a family of random variables $(N_{x,\alpha})_{x\in\N}$ by the conditions that $N_{x,\alpha}$ is supported on $\{1,\dots,x\}$, and that $\P(N_{x,\alpha}=k)$ is proportional to $\alpha(k)$.\footnote{More precisely, $\P(N_{x,\alpha}=k)=\frac{\alpha(k)}{\alpha(1)+\alpha(2)+\dots+\alpha(x)}$.}

    This manuscript demonstrates that results derived from the Selberg--Delange Method provide a framework for proving mod-Poisson convergence of the $\omega(N_{x,\alpha})$ for a broad class of multiplicative functions $\alpha$. As a direct consequence, we recover, under different, but related, conditions (see Remark~\ref{rem:conditions in Elboim--Gorodetsky}) the first part of \cite[Theorem 1.1]{Elboim-Gorodetsky}, which establishes a Central Limit Theorem for the $\omega(N_{x,\alpha})$.

    \bigskip
    We will now state the main result of our text.
    \begin{theorem}[Main result, shortened version of Theorem \ref{thm:main}]\label{thm:main-introductory}
    	Let $\alpha:\N\to[0,\infty)$ be a multiplicative function that is not identically equal to $0$. If $\alpha$ is \emph{admissible++}, a term that will be defined later in the text, see Definition~\ref{def:admissible-multiplicative-functions} and Definition~\ref{def:admissible++}, and $(N_{x,\alpha})_{x\in\N}$ is a family of random variables such that $\P(N_{x,\alpha}\in\set{1,\dots, x})=1$ and such that $\P(N_{x,\alpha}=k)$ is proportional to $\alpha(k)$ for $k\in\set{1,\dots, x}$, then the number of distinct prime factors of $N_{x,\alpha}$, denoted by $\omega(N_{x,\alpha})$, converges mod-Poisson (see Definition \ref{def:mod-phi-convergence}) as $x\to\infty$ on the entire complex plane.
    \end{theorem}

	A more general result can also be obtained for the family of random variables $(g(N_{x,\alpha}))_{x\in\N}$, where $g:\N\to\C$ is chosen within a broader class of additive functions, see section \ref{sect:general-additive}, specifically Theorem \ref{thm:general-additive}. However, even for mundane choices of $g$, such as choosing $g$ to be the function $\Omega$ which counts the number of prime factors of an integer with their multiplicities, the domain on which mod-Poisson convergence occurs may be smaller than $\C$. For ease of exposition, we thus mainly treat the case $g=\omega$.

	The structure of our text is as follows. After this Introduction, we first introduce the main tools from analytic number theory that we will use in order to establish Theorem \ref{thm:main-introductory} in section~\ref{sect:analytic-number-theory}. Since our main result includes the notion of mod-$\phi$ convergence, we recall its Definition in section~\ref{sect:mod-phi}. We will then prove Theorem \ref{thm:main-introductory} in section \ref{sect:main-result}. Afterwards, in section~\ref{sect:Consequences}, we derive a variety of consequences for the $\omega(N_{x,\alpha})$, namely a Central Limit Theorem and precise estimates for large deviations.  
	
	\section{Main result}\label{sect:results}
    In this section, we discuss the main result of this text. First, the notion of admissible++ functions, namely those multiplicative functions for which the main results hold, will be introduced in section~\ref{sect:analytic-number-theory}. We then recall the Definition of mod-$\phi$ convergence in section~\ref{sect:mod-phi}. Finally, we are in a position to give a precise statement of the main result in section~\ref{sect:main-result}.
    
    \subsection{Admissible functions}\label{sect:analytic-number-theory}
    The aim of this text, initially outlined in Theorem~\ref{thm:main-introductory} and detailed in section~\ref{sect:results}, is to prove mod-Poisson convergence for the number of prime factors of a random integer whose distribution is that given by a multiplicative function, as described in the Introduction.
    
    In order to do so, we need to establish asymptotic estimates for the moment generating function of the number of prime factors of that random variable. We will use the Selberg--Delange Method from analytic number theory, laid out in detail in \cite{Tenenbaum}, which establishes conditions under which precise asymptotic estimates for the sum $\sum_{k=1}^x f(k)$ as $x\to\infty$ for a multiplicative function $f$ can be made. 
    
    We now define the class of \emph{admissible} functions, which we need to state the results from \cite{Tenenbaum}
    \begin{definition}[Admissibility of functions]\label{def:admissible-multiplicative-functions}
        Let $f:\N\to\C$ be a function. We say that $f$ is \emph{admissible with average value $\rho\in\C$} if and only if there exist $c_0,M\in(0,\infty)$ and a $\delta\in(0,1]$ such that the function 
        \begin{equation}\label{eq:def-G}
            G(s)\define \sum_{n\in\N} \frac{f(n)}{n^s} \zeta(s)^{-\rho},
        \end{equation}
        with $\zeta$ denoting the Riemann Zeta function, is well-defined for some $s\in\C$,\footnote{Note that convergence of the Dirichlet series $\sum_{n\in\N} \frac{f(n)}{n^s}$ for some $s\in\C$ implies convergence on the entire half-plane to the right of $s$, see \cite[Section 9.12]{Titchmarsh}.} and can be continued to an analytic function on the domain\footnote{For a complex number $s$, we denote its real part by $\operatorname{Re}(s)$ and its imaginary part by $\operatorname{Im}(s)$.}
        \begin{equation}\label{eq:domain}
            \set*{s\in\C: \operatorname{Re}(s)\ge 1 - \frac{c_0}{1+\max\set{0, \ln\abs{\operatorname{Im}(s)}}}},
        \end{equation}
        and, on this domain, satisfies the inequality
        \begin{equation}\label{eq:inequality}
            \abs{G(s)}\le M(1+\abs{\operatorname{Im}(s)})^{1-\delta}.
        \end{equation}
    \end{definition}
    A function that is both admissible and multiplicative will be called an \emph{admissible multiplicative function}.

    \begin{example}[Examples of admissible functions]\label{example:admissible}
        The prototypical example of an admissible multiplicative function with average value $\rho$ is the multiplicative function $\tau_\rho$ whose Dirichlet series is the $\rho$th power of the Riemann Zeta function, that is,
        \begin{equation*}
            \sum_{n\in\N}\frac{\tau_\rho(n)}{n^s} = \zeta(s)^\rho.
        \end{equation*}
        A function is determined by its Dirichlet series whenever the latter converges for any $s\in\C$, in particular, one can show that $\tau_\rho$ satisfies
        \begin{equation*}
            \tau_\rho(p^\nu) = (\rho+\nu-1)(\rho+\nu-2)\dots(\rho-1), \qquad\text{for all }\nu\in\N, \text{ and }p \text{ prime}.
        \end{equation*}
        We thus have, by construction, that the function $G$ given by \eqref{eq:def-G} for $f=\tau_\rho$ is given by
        \begin{equation*}
            G(s)=1,
        \end{equation*}
        which can be analytically continued to the entire complex plane, and satisfies \eqref{eq:inequality} for $M=\delta=1$. 

        Another example to consider is the Euler totient function $\varphi$, mapping a natural number $n$ to the number of natural numbers less or equal than $n$ which are co-prime to $n$. In particular, we have $\varphi(p^k)= p^{k}-p^{k-1}$ for any prime $p$ and any $k\in\N$. Euler's totient function is multiplicative, and for any $s\in\C$ with $\operatorname{Re}(s)>2$, we have
        \begin{equation*}
            \sum_{n\in\N}\frac{\varphi(n)}{n^s}=\frac{\zeta(s-1)}{\zeta(s)}.
        \end{equation*}
        We see that $\varphi$ cannot be admissible, since $\zeta(s-1)$ has a pole at $s=2$. However, the multiplicative function $\tilde\varphi:n\mapsto\frac{\varphi(n)}{n}$ is admissible, since
        \begin{equation*}
            \sum_{n\in\N}\frac{\varphi(n)}{n^s} = \frac{\zeta(s)}{\zeta(s+1)},
        \end{equation*}
        so by compensating with $\rho=1$, we get
        \begin{equation*}
            \zeta(s)^{-1}\sum_{n\in\N}\frac{\varphi(n)}{n^s} = \frac{1}{\zeta(s+1)}, 
        \end{equation*}
        which is bounded on any half-plane of the form $\operatorname{Re}(s)>\varepsilon$ for a fixed $\varepsilon>0$.
    \end{example}

    We will need a stronger notion of admissibility which has the property that if a non-negative multiplicative function $f$ satisfies this stronger property, then also the multiplicative function $f_y: k\mapsto y^{\omega(k)} f(k)$ satisfies this property. We will introduce such a notion now.

    \begin{definition}[Admissibility++]\label{def:admissible++}
        Let $f:\N\to\C$ be a multiplicative function. We will say that $f$ is \emph{admissible++$(\rho, c_0, M)$} for $\rho\in\C$, $c_0\in(0,1)$, $M\in(0,\infty)$ if and only if $f$ is admissible with average value $\rho$ for those constants $c_0, M$, as well as for $\delta=1$ in Definition~\ref{def:admissible-multiplicative-functions}, and furthermore if there exists an $s\in\C$ such that 
        \begin{equation}\label{eq:Dirichlet-absolute-convergence}
            \sum_{n\in\N}\abs*{\frac{f(n)}{n^s}}<\infty,
        \end{equation}
        and furthermore 
        \begin{equation}\label{eq:sum-of-squares}
            \sum_{p\text{ prime}} \left(\sum_{k=1}^\infty\frac{\abs{f(p^k)}}{p^{k(1-c_0)}}\right)^2<\infty.
        \end{equation}
    \end{definition}

    \begin{remark}[Discussion of admissibility++]
        Condition \eqref{eq:Dirichlet-absolute-convergence} guarantees that the Dirichlet series of $f$ converges absolutely on the half-plane to the right of $s$, and furthermore that $f$ can be written there as an Euler product, see \cite[Section 9.13]{Titchmarsh} and \cite[Theorem 1.3]{Tenenbaum}. The idea behind admissibility++ is that condition \eqref{eq:sum-of-squares} guarantees that the Euler product for $f_y$ instead of $f$ can be extended to the domain \eqref{eq:domain}, and that a bound similar to \eqref{eq:inequality} is still satisfied. 
    \end{remark}

    \begin{remark}[The condition that $\delta=1$]
        In our proof of Proposition~\ref{lem:admissible-function-classes}, a term used to bound $\sum_{n\in\N}\frac{f_y(n)}{n^s}\zeta(s)^{-y\rho}$ involves the $y$-th power of $G$ as defined in \eqref{eq:def-G}, so the inequality \eqref{eq:inequality} would turn into
        \begin{equation*}
            \abs*{\sum_{n\in\N}\frac{f_y(n)}{n^s}\zeta(s)^{-y\rho}}\le M(y) (1+\abs{\operatorname{Im}(\delta)})^{(1-\delta)\operatorname{Re}(y)}
        \end{equation*}
        for some constant $M(y)$. The exponent $(1-\delta)\operatorname{Re}(y)$ makes it so that in the proof of our main result, Theorem~\ref{thm:main}, we may not obtain mod-Poisson convergence on the entire complex plane. We therefore restrict our analysis to the case $\delta=1$. With a more careful approach, one may be able to formulate a more general result, where one obtains mod-Poisson convergence only on a subset of the complex plane for $\delta<1$. 
    \end{remark}

    \begin{remark}[Is the notion of admissibility++ necessary?]
         We would ideally like to work directly with admissibility instead of admissibility++. However, it was not clear to me whether admissibility of a multiplicative function $f$ implies admissibility of the function $f_y$. Therefore, the notion of admissibility++ was introduced here.
    \end{remark}
    
    \begin{remark}[Uniformity of the convergence in \eqref{eq:sum-of-squares}]\label{rem:uniformity Fp}
        Note that for every $s\in\C$ with $\operatorname{Re}(s)\ge 1-c_0$, we have
        \begin{equation*}
            0\le\frac{\abs{f(p^k)}}{\abs{p^{k s}}}\le\frac{\abs{f(p^k)}}{p^{k(1-c_0)}}.
        \end{equation*}
        In particular, it follows that condition \eqref{eq:sum-of-squares} implies uniformly absolute convergence of $\sum_{p\text{ prime}} F_p(s)^2$ for all $s\in\C$ which satisfy $\operatorname{Re}(s)\ge 1-c_0$.
    \end{remark}
    
	We now observe that, given an admissible++$(\rho, c_0, M)$ multiplicative function $f$, we obtain another admissible++ multiplicative function by multiplying $f$ with the function mapping $k$ to $y^{\omega(k)}$ for a fixed $y\in\C$.

	 \begin{proposition}[Classes of admissible functions]\label{lem:admissible-function-classes}
		Let $f:\N\to\C$ be a multiplicative function that is admissible++$(\rho, c_0, M)$ for $\rho, c_0, M$ as in Definition \ref{def:admissible++}. Then, for any $y\in\C$, the function $f_y:\N\to\C$ defined by $f_y(n)=y^{\omega(n)} f(x)$ is a multiplicative function that is
        \begin{equation*}
            \text{admissible++}(y\rho, c_0, M(y)),
        \end{equation*}
        where the $M(y)$ can be chosen to depend continuously on $y$.
	\end{proposition}

    \begin{proof}
        First, recall that 
        \begin{equation*}
            \omega(n)=O\left(\frac{\ln n}{\ln \ln n}\right)
        \end{equation*}
        as $n\to\infty$. Therefore, if $\sum_{n\in\N}\abs*{\frac{f(n)}{n^s}}<\infty$ for some $s\in\C$, the same will be true for $f_y$ instead of $f$ if we increase the real part of $s$ by $\varepsilon>0$. 

        Second, 
        \begin{equation*}
            \sum_{p\text{ prime}}\left(\sum_{k=1}^\infty \frac{\abs{f_y(p^k)}}{p^{k(1-c_0)}}\right)^2 = \abs{y}^2\sum_{p\text{ prime}}\left(\sum_{k=1}^\infty \frac{\abs{f(p^k)}}{p^{k(1-c_0)}}\right)^2<\infty.
        \end{equation*}
   
        Third, for all $s$ which lie in the half-plane to the right of that complex number for which the Dirichlet series of $f$ converges absolutely, we can expand out the Euler product
        \begin{equation}\label{eq:Gy-Euler}
            G_y(s)\define\zeta(s)^{-y\rho} \sum_{n\in\N} \frac{f_y(n)}{n^s} = \prod_{p\text{ prime}}(1-p^{-s})^{y\rho} (1+yF_p(s)),
        \end{equation}
        where
        \begin{equation*}
            F_p(s)\define\sum_{k=1}^\infty \frac{f(p^k)}{p^{ks}}.
        \end{equation*}
        If $1+y F_p(s)=0$ for any prime $p$, then this product is $0$ and there are no convergence issues. We thus assume from now on that $1+yF_p(s)\neq 0$ for all primes $p$. We split this product into two parts: The set $A$ of those primes $p$ for which $\abs{F_p(s)}\ge\frac 12$, and the set $B$ of and those primes $p$ for which $\abs{F_p(s)}<\frac 12$. Let us also write
        \begin{multline*}
        G_1(s)\define \prod_{p\text{ prime}} (1-p^{-s})^\rho (1+F_p(s)) \\= \prod_{p\in A} (1-p^{-s})^\rho (1+F_p(s))  \prod_{p\in B} (1-p^{-s})^\rho (1+F_p(s))\define G_1^{(A)}(s) G_1^{(B)}(s).
        \end{multline*}

        Then
        \begin{equation}\label{eq:Gy-decomposition}
            G_y(s) =(G_1^{(B)}(s))^y \prod_{p\in A}(1-p^{-s})^{y\rho}(1+y F_p(s)) \prod_{p\in B} \frac{1+y F_p(s)}{(1+F_p(s))^y} 
        \end{equation}
        By assumption, $\sum_{p\text{ prime}} \abs{F_p(s)}^2$ converges whenever $\operatorname{Re}(s)\ge 1-c_0$, thus in particular $F_p(s)\to 0$ as $p\to\infty$, which implies that the set $A=A(s)$ is finite. Furthermore, by Remark \ref{rem:uniformity Fp}, the convergence of $\sum_{p\text{ prime}}\abs{F_p(s)}^2$ is uniform in $s$, so that the union of all the $A(s)$ for $s$ in the half-plane to the right of $1-c_0$ is finite. Thus the product $\prod_{p\in A} (1-p^{-s})^\rho (1+F_p(s))$ is uniformly-in-$s$ a finite product and can be bounded uniformly-in-$s$ by a constant which depends continuously on $y$.

        Similarly, recalling that $\abs{G_1(s)}\le M$ and that $G_1^{(B)}(s) = \frac{G_1(s)}{G_1^{(A)}(s)}$, with $G_1^{(A)}(s)$ being a uniformly-in-$s$ finite product, we can uniformly-in-$s$ bound $G_1^{(B)}(s)$ by a constant that depends continuously on $y$ and $M$. 
        
        Consider now the function $f:\C\setminus\set{-1}\to\C, a\mapsto\frac{1+ya}{(1+a)^y}$. We have $f(0)=1, f'(0)=0, f''(0)=y(1-y)$. Therefore, by Taylor's Theorem, 
        \begin{equation*}
            \abs*{\frac{1+ya}{(1+a)^y}-1 - \frac{a^2}2 y (1-y)}\le C\abs{a}^3
        \end{equation*}
        for all $a$ with $\abs{a}\le\frac 12$ and a constant $C$ that depends continuously on $y$. Since we have chosen the set $B$ to be comprised of those $p$ for which $\abs{F_p(s)}\le\frac 12$, we conclude
        \begin{multline*}
            \prod_{p\in B} \frac{1+y F_p(s)}{(1+F_p(s))^y} =\prod_{p\in B} \left(1+F_p(s)^2 y(1-y)+O_{p\to\infty}(F_p(s)^3)\right)\\\le\exp\left(\abs{y(1-y)}\sum_{p\text{ prime}} \abs{F_p(s)}^2+O_{p\to\infty}(F_p(s)^3)\right),
        \end{multline*}
        where the constant implicit in the $O_{p\to\infty}$ depends continuously on $y$. 

        In conclusion, we have convergence of the Euler product \eqref{eq:Gy-Euler}, since we have shown convergence of every term in \eqref{eq:Gy-decomposition}. Furthermore, we can uniformly-in-$s$ bound every term in \eqref{eq:Gy-decomposition} by a constant which depends continuously on $y$ and $M$. In particular, we can define $G_y$ for those $s$ in the domain \eqref{eq:domain} even if the Dirichlet series $\frac{f_y(n)}{n^s}$ is not convergent for all those $s$. This concludes the proof that $f_y$ is admissible for $\delta=1$, the $c_0$ as given, and a constant $M(y)$ that depends continuously on $y$ and $M$. 

        Following our argument, we see that the convergence of our constructed Euler product \eqref{eq:Gy-decomposition} is locally uniform. Since the locally uniform convergence of complex analytic functions is again complex analytic, we have established admissibility of $f_y$.
    \end{proof}

	If a multiplicative function $f$ is admissible and its modulus is bounded by an admissible function, we can get precise asymptotic estimates as $x\to\infty$ on its running sums $\sum_{k=1}^x f(k)$, as established by the following result.
    \begin{theorem}[{Corollary of \cite[Theorem II.5.2]{Tenenbaum}}]\label{thm:Tenenbaum-main}
        Let $f:\N\to\C$ be an admissible function with average value $\rho$. Moreover, assume that there exists an admissible function $\tilde f:\N\to[0,\infty)$ such that $\abs{f}\le \tilde f$. Then
        \begin{equation}\label{eq:main-asymptotic}
            \sum_{k=1}^x f(k) =  x (\ln x)^{\rho-1}\left(\lambda_0(f)+O\left(\frac{1}{\ln x}\right)\right), \qquad\text{as }x\to\infty,
        \end{equation}
        where $\lambda_0$ is well-defined through the following formula if $\rho\in\C\setminus\Z_{\le 0}$
        \begin{equation}\label{eq:lambda-0}
            \lambda_0(f)\define\frac{1}{\Gamma(\rho)}\prod_{p\text{ prime}}\left(\left(1-\frac 1p\right)^\rho\sum_{k=0}^\infty \frac{f(p^k)}{p^{k}}\right),
        \end{equation}
        and where $\lambda_0(f)\define 0$ for $\rho\in\Z_{\le 0}$.

        The constant implicit in the $O$ of \eqref{eq:main-asymptotic} depends only on $c_0, \delta$ from Definition \ref{def:admissible-multiplicative-functions}, as well as locally uniformly on $\rho$ and $M$ from Definition \ref{def:admissible-multiplicative-functions}.\footnote{Locally uniformly means that for all $\rho$ and $M$ varying in a compact set, we can pick the same implicit constant.}
    \end{theorem}

    \begin{remark}[Remark on Theorem \ref{thm:Tenenbaum-main}]
        Theorem \ref{thm:Tenenbaum-main} as stated here is a Corollary of \cite[Theorem II.5.2]{Tenenbaum}. The latter gives a full asymptotic expansion for $\sum_{k=1}^x f(k)$. However, for our purposes, the first term of that expansion is enough.
    \end{remark}

	\subsection{Recall of mod-\texorpdfstring{$\phi$}{phi} convergence}\label{sect:mod-phi}
	Mod-$\phi$ convergence, originally introduced in \cite{FMN}, is a strengthening of the notion of convergence in distribution towards a normal distribution. A special case of mod-$\phi$ convergence is mod-Poisson convergence. 
	
	As the main result of this text is the mod-Poisson convergence of the $\omega(N_{x,\alpha})$ as defined in Theorem \ref{thm:main-introductory}, we will recall here the Definition of mod-$\phi$ and mod-Poisson convergence.
	\begin{definition}[{\cite[Definition 1.1]{FMN}}]\label{def:mod-phi-convergence}
		Let $(X_n)_{n\in\N}$ be a sequence of real-valued random variables and let $c,d\in[-\infty,\infty]$ satisfy $c<d$. Assume that for $z\in\kS_{(c,d)}\define\set{z\in\C: c < \operatorname{Re} z < d}$, the moment generating functions $\phi_n(z)\define\E\left(\e^{z X_n}\right)$ are finite and well-defined. (In particular, each of the $X_n$ have finite moments of all orders.) Fix furthermore an infinitely divisible distribution $\phi$ on $\R$. By the Lévy-Khintchine formula, such a $\phi$ which has a well-defined and finite cumulant generating function
		\begin{equation*}
			\eta(z)\define\ln\left(\int_\R \e^{zx}\,\mathrm d\phi(x)\right)
		\end{equation*}
		on $\kS_{(c,d)}$. Finally, assume the existence of an analytic function $\psi:\kS_{(c,d)}\to\C$ with $0~\not\in~\psi((c,d))$ such that 
		\begin{equation}\label{eq:psi-n-convergence}
			\psi_n\define\exp(-t_n\eta)\phi_n\to\psi\quad\text{in }L^\infty_{\text{loc}}(\kS_{(c,d)})\text{ as }n\to\infty,
		\end{equation}
		where $(t_n)_{n\in\N}$ is a sequence of real numbers going to $\infty$. Under the above conditions, the $(X_n)_{n\in\N}$ are said to \emph{converge mod-$\phi$} on $\kS_{(c,d)}$ with parameters $(t_n)_{n\in\N}$ and limiting function $\psi$.
	\end{definition}

	We are now in a position to define mod-Poisson convergence.
	\begin{definition}[Mod-Poisson convergence]
		If a sequence of random variables converges mod-$\phi$, where $\phi$ is a Poisson distribution on $\R$ with parameter $1$, which means that $\eta(z)=\e^z-1$, then we say that the sequence of random variables converges \emph{mod-Poisson}.
	\end{definition}

    \subsection{Statement of main result and its proof}\label{sect:main-result}
    We are now ready to state the main result of our manuscript. The following Theorem is a more detailed rendition of Theorem~\ref{thm:main-introductory}.
    \begin{theorem}\label{thm:main}
        Let $\alpha:\N\to[0,\infty)$ be an admissible++$(\rho, c_0, M)$ function for some $\rho\in[0,\infty)$, $c_0\in(0,1)$ and $M\in(0,\infty)$. Assume that $\alpha$ is not identically equal to $0$ (which, since $\alpha$ is multiplicative, implies $\alpha(1)=1$). Let $(N_{x,\alpha})_{x\in\N}$ be a family of random variables such that
        \begin{equation*}
        	\P(N_{x,\alpha}=k) = \frac{\mathds 1_{k\in\{1,\dots,x\}} \alpha(k)}{\sum_{j=1}^x \alpha(j)}, \qquad \text{for all }k\in\N.
        \end{equation*}
        Then the family of random variables $(\omega(N_{x,\alpha}))_{x\in\R}$ converges mod-Poisson on the entire complex plane with limiting function
        \begin{equation*}
            \psi(z) = \frac{\lambda_0(\alpha_{\e^z})}{\lambda_0(\alpha)},
        \end{equation*}
        where $\alpha_y(x)\define y^{\omega(x)} \alpha(x)$ for $y\in\C, x\in\N$, and $\lambda_0(\alpha_{\e^z})$ is given by \eqref{eq:lambda-0} for $\e^z\rho\in\C\setminus\Z_{\le 0}$, and by $0$ for $\e^z\rho\in\Z_{\le 0}$. 

        The parameters of the convergence are $t_x=\rho\ln\ln x$ and the speed of convergence is $O(1/\ln x)$.
    \end{theorem}

	\begin{proof}[Proof of Theorem \ref{thm:main}]
		Since $\alpha$ is admissible with average value $\rho$ by assumption and since it is non-negative, from  Theorem~\ref{thm:Tenenbaum-main}, we get
		\begin{equation}\label{eq:estimation-denominator}
			\sum_{n=1}^x \alpha(n) = x (\ln x)^{\rho-1} (\lambda_0(\alpha) + O(1/\ln x)), \qquad\text{as }x\to\infty.
		\end{equation}
	
		Now, let $y\in\C$ and let $\alpha_y$ be defined as in Lemma \ref{lem:admissible-function-classes}. Since $\alpha$ is admissible++, by Proposition \ref{lem:admissible-function-classes}, $\alpha_y$ is admissible with average value $y \rho$. We note furthermore that $\alpha_y$ is admissible for the same $c_0$ in Definition~\ref{def:admissible-multiplicative-functions} as $\alpha$. Thus, by Theorem~\ref{thm:Tenenbaum-main}, noting that $\abs{y^{\omega(n)}\alpha}=\abs{y}^{\omega(n)}\alpha$ is admissible too, we get
		\begin{equation}\label{eq:estimation-numerator}
			\sum_{n=1}^x y^{\omega(n)} \alpha(n) = x (\ln x)^{\rho y-1} \left(\lambda_0(\alpha_y) + O(1/\ln x)\right), \qquad\text{as }x\to\infty,
		\end{equation}
		where by convention we set $\lambda_0(\alpha_y)$ equal to $0$ if $\rho y\in\Z_{\le 0}$. 

        The idea of the proof is to now combine \eqref{eq:estimation-denominator} and \eqref{eq:estimation-numerator} to obtain asymptotic estimates for the moment-generating function $\E\left(y^{\omega(N_{x,\alpha})}\right)$, as will be done in \eqref{eq:main-calc} below. Note that mod-$\phi$ convergence requires not only pointwise-in-$y$ convergence, but it requires locally uniform convergence (see equation \eqref{eq:psi-n-convergence}). As already mentioned, the $c_0$ from Definition \ref{def:admissible-multiplicative-functions} are the same for $\alpha_y$ as they are for $\alpha$, and furthermore $\delta=1$ for both. The constants implicit in the $O$ can thus be chosen the same for $\alpha$ and $\alpha_y$, as long as $y$ varies in a compact set, as the constant implicit in \eqref{eq:main-asymptotic} depend only on $c_0,\delta$ and locally uniformly on the average value of the multiplicative function, as well as $M$, but the latter varies continuously depending on $y$ by Proposition~\ref{lem:admissible-function-classes}.

		We thus conclude that
		\begin{equation}\label{eq:main-calc}\begin{split}
			\E\left(y^{\omega(N_{x,\alpha})}\right) &= \frac{\sum_{k=1}^x y^{\omega(k)} \alpha(k)}{\sum_{k=1}^x \alpha(k)} \\
			&= \frac{x (\ln x)^{\rho y-1} \left(\lambda_0(\alpha_y) + O\left(1/\ln x\right)\right)}{x (\ln x)^{\rho-1}\left( \lambda_0(\alpha)+ O\left(1/\ln x\right)\right)}\\ 
			&= \frac{(\ln x)^{\rho(y-1)}}{\lambda_0(\alpha)} \left(\lambda_0(\alpha_y)+O\left(1/\ln x\right)\right)\left(1+O\left(1/\ln x\right)\right)\\ 
			&= (\ln x)^{\rho(y-1)}\frac{\lambda_0(\alpha_y)}{\lambda_0(\alpha)}\left(1+O\left(1/\ln x\right)\right),
		\end{split}\end{equation}
		where for any fixed $B\in(0,\infty)$, the constant implicit in the $O$ can be chosen uniformly for all $y\in\C$ with $\abs y\le B$.
		
		It follows that for $z\in\C$, and any $\varepsilon\in(0,\infty)$,
		\begin{equation*}
			\E\left(\e^{z\omega(N_{x,\alpha})}\right) =  \e^{\rho (\ln \ln x) (\e^z-1)} \frac{\lambda_0(\alpha_y)}{\lambda_0(\alpha)}\left(1+O\left(1/\ln x\right)\right),
		\end{equation*}
		where for any fixed $B\in(0,\infty)$, the constant implicit in the $O$ can be chosen uniformly for all $z\in\C$ for which $\lvert z\rvert\le B$.
		
		Recall that the Laplace transform of the Poisson distribution with parameter $1$ is given by
		\begin{equation*}
			z\mapsto\e^{\e^z-1}.
		\end{equation*}
		Therefore, we have mod-Poisson convergence of the $\omega(N_{x,\alpha})$ as $x\to\infty$ with limiting function 
		\begin{equation*}\psi(z)=\frac{\lambda_0(\alpha_{\e^z})}{\lambda_0(\alpha)}\end{equation*}
		and parameters $t_x = \rho\ln\ln x$ on the entire complex plane. The speed of convergence is $O\left(1/\ln x\right)$. 
	\end{proof}

	\subsection{Generalization to additive functions other than \texorpdfstring{$\omega$}{omega}}\label{sect:general-additive}
	In this section we will discuss generalizations of Theorem~\ref{thm:main} to $g(N_{x,\alpha})$, where $g$ is some additive function. Theorem~\ref{thm:main} corresponds to the choice $g=\omega$.
	
	\begin{definition}[Additive function]
		A function $g:\N\to\C$ is called \emph{additive} if, for any co-prime integers $n,m\in\N$, we have
		\begin{equation*}
			g(nm)=g(n)+g(m).
		\end{equation*}
	\end{definition}
    For example, the functions $\omega, \Omega$ (see Example \ref{example:Omega}), the logarithm restricted to natural numbers and the function mapping an integer to the sum of the distinct prime divisors of that integer are additive functions.
	In order to establish mod-Poisson convergence of $g(N_{x,\alpha})$ for an additive function $g$, we need asymptotic estimates as $\N\ni x\to\infty$ for the sum
	\begin{equation*}
		\sum_{n=1}^x y^{g(n)}\alpha(n) 
	\end{equation*}
	for $y$ in at least a subset of the complex plane. In principle, it would be best if we could get such asymptotic estimates for all $y\in\C$, however, as we will see in Example \ref{example:Omega} below, even for fairly mundane choices of $g$, such as choosing $g=\Omega$, where $\Omega$ is the function counting the number of prime factors of an integer with multiplicity, will lead to $n\mapsto y^{g(n)} \alpha(n)$ being non-admissible for some $y\in\C$.
	
	We briefly discuss how to adjust the proof of Theorem~\ref{thm:main} to this more general case. In order to obtain \eqref{eq:estimation-denominator} as well as \eqref{eq:estimation-numerator}, we need that $n\mapsto y^{g(n)} \alpha(n)$ is admissible for all $y$ of interest, which for mod-$\phi$ convergence means for all $y$ in some set $\set{\exp(z):z\in\kS_{(c,d)}}$, borrowing notation from Definition \ref{def:mod-phi-convergence}. We furthermore need that the constant implicit in the $O$ obtained from Theorem~\ref{thm:Tenenbaum-main} varies locally uniformly with $y$. 
  	
	We summarize the previous discussion in the following Theorem.
	\begin{theorem}[Main result, most general version]\label{thm:general-additive}
		Let $\alpha:\N\to[0,\infty)$ be a multiplicative function which is not identically equal to $0$. Furthermore, let $c,d\in[-\infty,\infty]$ satisfy $c<d$, and let $\mathcal S_{(c,d)}\subset\C$ denote the set $\set{z\in\C:c<\operatorname{Re}z<d}$. Fix an additive function $g:\N\to[0,\infty)$ and assume that all of the following hold:
        \begin{enumerate}
            \item there exist $\rho\in[0,\infty)$, $\delta\in(0,1]$ and $c_0\in(0,1)$, as well as a function $M:\kS_{(c,d)}\to(0,\infty)$ such that for every $z\in\kS_{(c,d)}$, the function $\alpha_{\e^z}:\N\to\C, n\mapsto \e^{zg(n)} \alpha(n)$, is admissible for that $\delta, c_0$ and for $M=M(z)$ with average value $\rho \e^z$;
            \item for any compact set $C\subset\kS_{(c,d)}$, $\sup_{z\in C}M(z)<\infty$.
        \end{enumerate}
        Let $(N_{x,\alpha})_{x\in\N}$ be a family of random variables such that
        \begin{equation*}
        	\P(N_{x,\alpha}=k) = \frac{\mathds 1_{k\in\{1,\dots,x\}} \alpha(k)}{\sum_{j=1}^x \alpha(j)}, \qquad \text{for all }k\in\N.
        \end{equation*}
        Then the family of random variables $(g(N_{x,\alpha}))_{x\in\R}$ converges mod-Poisson on $\kS_{(c,d)}$ with limiting function
        \begin{equation*}
            \psi(z) = \frac{\lambda_0(\alpha_{\e^z})}{\lambda_0(\alpha)},
        \end{equation*}
        where $\lambda_0(\alpha_{\e^z})$ is given by \eqref{eq:lambda-0} for $\e^z\rho\in\C\setminus\Z_{\le 0}$, and by $0$ for $\e^z\rho\in\Z_{\le 0}$. 

        The parameters of the convergence are $t_x=\rho\ln\ln x$ and the speed of convergence $O(1/\ln x)$.
	\end{theorem}

    We hope to illuminate this more general Theorem with the following example.

    \begin{example}[Prime factors counted with multiplicities]\label{example:Omega}
        We consider the function $g=\Omega$, with $\Omega(n)$ defined as the number of prime factors of $n\in\N$, counted with multiplicity. We choose, for simplicity, the multiplicative function $\alpha(n)=1$, $n\in\N$. For $y\in\C$ satisfying $\abs{y}<2$, the Dirichlet series
        $\sum_{n\in\N} \frac{y^{\Omega(n)}}{n^s}$ is absolutely convergent when $\operatorname{Re}(s)>1$ and we can thus expand it into the Euler product
        \begin{equation*}
            \sum_{n\in\N} \frac{y^{\Omega(n)}}{n^s}=\prod_{p\text{ prime}}\sum_{k=0}^\infty\frac{y^k}{p^{ks}}=\prod_{p\text{ prime}} \left(1-\frac{y}{p^s}\right)^{-1}.
        \end{equation*}
        In particular, choosing $\rho=y$,
        \begin{equation*}
            \zeta(s)^{-y}\sum_{n \in\N}\frac{y^{\Omega(n)}}{n^s} = \prod_{p\text{ prime}}(1-p^{-s})^y (1-yp^{-s})^{-1}.
        \end{equation*}
        This product has a pole whenever $y=p^{-s}$, which is the reason why the function $n\mapsto y^{\Omega(n)}$ is not admissible for all $y\in\C$. However, assuming that $\abs{y}\le 2^{-\operatorname{Re}s}$, we may take the logarithm to obtain
        \begin{equation*}
            \sum_{p\text{ prime}} y \ln(1-p^{-s}) - \ln(1-y p^{-s}) = \sum_{p\text{ prime}} -yp^{-s}+yp^{-s} + O_{p\to\infty}(p^{-2s}) = \sum_{p\text{ prime}}O_{p\to\infty}(p^{-2s}).
        \end{equation*}
        In particular, this is uniformly convergent on the half-plane $\set*{s\in\C:\operatorname{Re}(s)>\frac 12}$. Thus, we have admissibility on this half-plane assuming that $\abs{y}<\sqrt 2$. Furthermore, one can check that the constant $M$ in \eqref{eq:inequality} depends locally uniformly on $y$. This means that we can apply Theorem~\ref{thm:general-additive} for the strip $\kS_{(-\infty,\ln(2)/2)}$ and thus obtain mod-Poisson convergence for the $(\Omega(N_x))_{x\in\N}$, where each $N_x$ is a random variable distributed uniformly on $\set{1,2,\dots, x}$.  
        
    \end{example}

    \begin{remark}[Erd\H{o}s-Kac Theorem for $\Omega$]
        Based on the discussion in Example \ref{example:Omega} and the results of section \ref{sect:CLT}, we remark that we obtain a Central Limit Theorem for the number of prime factors with multiplicity of a uniformly randomly chosen integer in $\set{1,\dots, x}$ as $x\to\infty$, that is, we have
        \begin{equation*}
            \frac{\Omega(N_x)-\ln\ln x}{\sqrt{\ln \ln x}} \to Z\text{ in distribution as }x\to\infty,
        \end{equation*}
    	where $N_x$ denotes a uniformly on $\set{1,2,\dots, x}$ distributed random variable, and $Z$ denotes a normally distributed random variable with mean $0$ and variance $1$.
        This is a well-known variant of the classical Erd\H{o}s-Kac Theorem.
    \end{remark}
	
	\section{Consequences}\label{sect:Consequences}
	This section outlines consequences that are derived from mod-Poisson convergence. In particular, in section \ref{sect:CLT} we obtain an extended Central Limit Theorem for the $\omega(N_{x,\alpha})$, thus recovering the same result as in \cite[Theorem 1.1, Part 1]{Elboim-Gorodetsky} under slightly different conditions (see Remark~\ref{rem:conditions in Elboim--Gorodetsky}, but we furthermore obtain precise large deviation estimates in section \ref{sect:Precise-large-deviation}. 

    \begin{remark}[Consequences for additive functions]
        For ease of exposition, we will consider in the following the behavior of $\omega(N_{x,\alpha})$ (borrowing notation from Theorem~\ref{thm:main}), for which we have mod-Poisson convergence if $\alpha$ is a multiplicative function satisfying the conditions of Theorem~\ref{thm:main}. One can straightforwardly apply instead Theorem~\ref{thm:general-additive} to obtain mod-Poisson convergence, and thus the results of this section, for $g(N_{x,\alpha})$, where $g$ is a general additive function satisfying the conditions of Theorem~\ref{thm:general-additive}.
    \end{remark}
	
	\subsection{Central Limit Theorem}\label{sect:CLT}
	Mod-$\phi$ convergence is a strengthening of convergence in distribution to a normal distribution, therefore, we can formulate a Central Limit Theorem for the converging variables. The main result of this section establishes such a Central Limit Theorem for the $\omega(N_{x,\alpha})$, see Corollary~\ref{cor:CLT} for details.
	\begin{proposition}[Central Limit Theorem for random variables that converge mod-$\phi$]\label{prop:CLT}
		If $(X_n)_{n\in\N}$ is a sequence of random variables converging mod-$\phi$ with $\eta$ and $(t_n)_{n\in\N}$ as in Definition \ref{def:mod-phi-convergence}, and if $\eta''(0)\neq 0$, then the sequence of random variables $(Y_n)_{n\in\N}$ given by 
		\begin{equation*}
			Y_n\define\frac{X_n - t_n \eta'(0)}{\sqrt{t_n \eta''(0)}}
		\end{equation*}
		converges in distribution to a standard normal distribution as $n\to\infty$.
	\end{proposition}
    \begin{remark}[Condition $\eta''(0)\neq 0$]
        Note that $\eta''(0)$ equals the variance of a random variable $X$ which is $\phi$-distributed (for details see Remark \ref{rem:strong-convexity}), thus the condition $\eta''(0)\neq 0$ is equivalent to the condition that the support of $\phi$ contains at least two distinct points. If this condition is not met, then instead $X_n-t_n \eta'(0)$ converges in distribution towards a random variable that is $\psi$-distributed, where $\psi$ is as in Definition \ref{def:mod-phi-convergence}, by a direct application of Lévy's continuity Theorem.
    \end{remark}
	Proposition \ref{prop:CLT} follows from a Taylor expansion of the cumulant generating function $\eta$, which we write out for illustrative purposes in appendix \ref{appendix:CLT}. Below is a stronger result for lattice distributions\footnote{Similar stronger results are also available for non-lattice distributions, see \cite[Theorem 4.8]{FMN}.}
	The Central Limit Theorem gives us asymptotics for the probability
	\begin{equation*}
		\P\left(X_n\ge t_n \eta'(0) + \sqrt{t_n\eta''(0)} y\right)
	\end{equation*}
	for some fixed $y\in\R$. However, we can get an extended Central Limit Theorem, where the $y$ is allowed to vary slightly with $n$, as explicated in the following result.  
	\begin{theorem}[Extended Central Limit Theorem, Corollary of {\cite[Theorem 3.9]{FMN}}]\label{thm:extended-CLT}
		Let $(X_n)_{n\in\N}$ be a sequence of random variables that converges mod-$\phi$ on any strip where $\phi$ is an infinitely divisible probability distribution on $\R$ whose support is contained in some set of the form $\gamma+\lambda\Z$ for some $\gamma\in\R, \lambda\in(0,\infty)$. Let $\eta$ and $(t_n)_{n\in\N}$ be as in Definition \ref{def:mod-phi-convergence}, and assume that $\eta''(0)\neq 0$. Then, for any sequence of real numbers $(y_n)_{n\in\N}$ with $y_n = o(t_n^{\frac 16})$ as $n\to\infty$, we have
		\begin{equation*}
			\P\left(X_n \ge t_n \eta'(0)+\sqrt{t_n \eta''(0)} y_n\right) = (1-\Phi(y_n))(1+o(1)),\qquad\text{as }n\to\infty,
		\end{equation*}
		where $\Phi(t)\define\frac{1}{\sqrt{2\pi}}\int_{-\infty}^t\e^{-\frac 12 x^2}\,\mathrm dx$ denotes the cumulative distribution function of a standard normal distribution.
	\end{theorem}
	Since the Poisson distribution is infinitely divisible, has its support contained in $\Z$ and since it is not concentrated on a single point (therefore $\eta''(0)\neq 0$), we thus get the following Corollary.
	\begin{corollary}[Extended Central Limit Theorem for the $N_{x,\alpha}$]\label{cor:CLT}
		Let $\alpha:\N\to[0,\infty)$ be an admissible multiplicative function with average value $\rho\in(0,\infty)$, and let $(N_{x,\alpha})_{x\in\N}$ be a family of random variables as in Theorem \ref{thm:main}. Then, for any sequence of real numbers $(y_x)_{x\in\N}$ with $y_x = o\left((\ln \ln x)^\frac 16\right)$ as $x\to\infty$, we have, with $\Phi$ defined as in Theorem \ref{thm:extended-CLT},
		\begin{equation*}
			\P\left(\omega(N_{x,\alpha})\ge \rho\ln \ln x + y_x\sqrt{\rho\ln \ln x}\right) = (1-\Phi(y_x))(1+o(1)),\qquad\text{as }x\to\infty.
		\end{equation*}
		In particular, the random variables
		\begin{equation*}
			\frac{\omega(N_{x,\alpha})-\rho\ln\ln x}{\sqrt{\rho\ln \ln x}}
		\end{equation*}
		converge in distribution to a standard normal distribution as $x\to\infty$.
	\end{corollary}

    \begin{remark}[Conditions in \cite{Elboim-Gorodetsky}]\label{rem:conditions in Elboim--Gorodetsky}
        The results in \cite{Elboim-Gorodetsky} employ a more recent formulation of the Selberg--Delange Method for non-negative multiplicative functions, namely \cite[Theorem 2.1]{BretecheTenenbaum}. The conditions of \cite[Theorem 2.1]{BretecheTenenbaum} are that the multiplicative function $f$ under study satisfies 
        \begin{equation*}
            \sum_{p\le x}f(p)\ln(p) = \rho x + O\left(\frac{x}{(\ln x)^A}\right), \qquad\text{ as }x\to\infty
        \end{equation*}
        for some $A,\rho\in(0,\infty)$, which may be seen as a condition related to admissibility. Furthermore there is a second-order condition related to \ref{eq:sum-of-squares}.
    \end{remark}
	
	\subsection{Precise large deviations}\label{sect:Precise-large-deviation}
	Mod-$\phi$ convergence implies precise large deviation principles. The main result of this section will be a precise large deviation result for the $\omega(N_{x,\alpha})$, see Corollary \ref{cor:large-deviation}.
	
	In order to motivate this result, we recall first the basic theory of large deviations, most notably, Cramér's large deviation principle. 
	
	\begin{theorem}[Cramér's large deviation principle]
		Let $(Y_n)_{n\in\N}$ be a sequence of independent, identically distributed real random variables. We consider the cumulant generating function $\Lambda(t)\define\ln\E\left(\e^{t Y_1}\right)$, which we assume to be finite and well-defined for all $t\in\R$.\footnote{In particular, the random variables $Y_n$ have finite moments of all orders.} We now define the \emph{Fenchel-Legendre transform} of the cumulant generating function as
		\begin{equation*}
			\Lambda^*(s) = \sup_{t\in\R} ts - \Lambda(t), \qquad\text{for }s\in\R. 
		\end{equation*}
		Then, for all $s>\E(Y_1)$, with $X_n\define Y_1+\dots+Y_n$,
		\begin{equation*}
			\lim_{n\to\infty} \frac 1n \P(X_n \ge sn) = -\Lambda^*(s).
		\end{equation*}
	\end{theorem}
	Large deviation principles say, intuitively, that, for some sequence of random variables $(X_n)_{n\in\N}$, we have $\P(X_n\ge sn) = \exp\left(-n \Lambda^*(s) + o(n)\right)$ for $s>\E(X_1)$ fixed as $n\to\infty$. We would like to establish \emph{precise} large deviations, meaning that we want better control of the $o(n)$ term.
	
	Precise large deviations can be obtained for random variables that converge mod-$\phi$, as illustrated in the following result.
	\begin{theorem}[{\cite[Theorem 3.4]{FMN}}]\label{thm:large-deviation}
		Let $\phi$ be an infinitely divisible probability distribution on $\R$ whose support is contained in some $\gamma+\lambda\Z$ for some $\gamma\in\R$ and $\lambda\in(0,\infty)$, and let $(X_n)_{n\in\N}$ be a sequence of random variables converging mod-$\phi$ on $\kS_{(c,d)}$ with $c<0<d$ and parameters $(t_n)_{n\in\N}$ with limiting function $\psi$ (borrowing the notation from Definition \ref{def:mod-phi-convergence}). Let $\eta$ be as in Definition \ref{def:mod-phi-convergence} and assume $\eta''(0)\neq0$. Then, for $s\in(\eta'(c),\eta'(d))$,
		\begin{equation*}
			\P(X_n\ge t_n s)= \exp(-t_n \eta^*(s)) \frac{\psi(h)}{1-\e^{-h}} (1+o(1)), \qquad\text{as }n\to\infty,
		\end{equation*}
		where 
		\begin{equation*}
			\eta^*(s)\define\sup_{t\in\R} ts - \eta(t), \qquad\text{for }s\in\R,
		\end{equation*}
		and $h$ is defined through the implicit equation $\eta'(h)=s$.
	\end{theorem}
	\begin{remark}[Well-definedness of $h$]\label{rem:strong-convexity}
		Consider $\phi,\eta$ as in Definition \ref{def:mod-phi-convergence}, and let $X$ be a $\phi$-distributed random variable. Then note that, for $z\in\R$,
		\begin{equation*}
			\eta''(z) = \frac{\E\left(X^2 \e^{z X}\right)\E\left(\e^{z X}\right) - \E\left(X\e^{zX}\right)^2}{\E\left(\e^{zX}\right)^2},
		\end{equation*}
		which is the variance of $X$ under the measure $\Q_z$ given by $\frac{\mathrm d\Q_z}{\mathrm d\P} = \frac{\e^{zX}}{\E(\e^{zX})}$. Therefore, $\eta$ restricted to $\R$ is a convex function. Furthermore, $\eta$ is strictly convex unless $\phi$ is supported on a single point.
		
		It follows that $h$ in Theorem \ref{thm:large-deviation} is well-defined and unique. Furthermore, we have $\eta^*(s) = s h - \eta(h)$ when $\eta'(h)=s$.
	\end{remark}
	\begin{remark}[The term $\frac{\psi(h)}{1-\e^{-h}}$]
		The term $\frac{\psi(h)}{1-\e^{-h}}$ in Theorem \ref{thm:large-deviation} is undefined when $h=0$. In this case, we implicitly define this term as $\psi'(0)$.
	\end{remark}
	
	\begin{corollary}[Precise large deviations for the $\omega(N_{x,\alpha})$]\label{cor:large-deviation}
		Let $\alpha:\N\to\R_+$ be an admissible multiplicative function with average value $\rho\in\C$. Then for $s\in(0,\infty)$, we have
		\begin{equation*}
			\P(N_{x,\alpha}\ge s\rho\ln\ln x) = \exp(-\rho\ln\ln x (1+s(\ln(s)-1))) \frac{\psi(\ln s)}{1-\frac 1s} (1+o(1)),\qquad\text{as }x\to\infty,
		\end{equation*}
		where, for $s=1$, we once again implicitly define $\frac{\psi(\ln s)}{1-\frac 1s}$ as $\psi'(0)$.
	\end{corollary}

	\section{Examples of admissible++ functions}\label{sect:examples}
	In this section, we give examples of families of multiplicative functions which are admissible++, and for which we can thus deduce a result of mod-Poisson convergence as in Theorem \ref{thm:main}.
	
	\subsection{\texorpdfstring{$\alpha(k)=\theta^{\omega(k)}$ for a $\theta>0$}{alpha(k)=theta pow omega(k) for a theta>0}}
	A well-known class of multiplicative functions, studied also in \cite{Elboim-Gorodetsky}, is given by $\alpha(n) = \theta^{\omega(n)}$ for a parameter $\theta\in(0,\infty)$. To show that such an $\alpha$ is admissible, by Proposition~\ref{lem:admissible-function-classes}, it is enough to prove that $\alpha\equiv 1$ is admissible. However, $\alpha\equiv 1$ has the Dirichlet series $\zeta(s)$, which directly shows that this $\alpha$ is admissible with average value $1$, as discussed in Example \ref{example:admissible}. Admissibility++ holds for $\rho=1$ and any $c_0<\frac 12$.
    
    \subsection{\texorpdfstring{$\alpha(p^s)=B^s$ for a $B\in(0,2)$}{alpha(ps)=Bs for a B in (0,2)}}
    Consider now a fixed $B\in(0,2)$ and let $\alpha:\N\to(0,\infty)$ be the multiplicative function which satisfies for all $s\in\Z_{\ge 0}$ as well as all primes $p$ the identity $\alpha(p^s) = B^s$. Then we have the Euler product
    \begin{multline}\label{eq:product-alpha-power-of-B}
        \sum_{n=1}^\infty \frac{\alpha(n)}{n^s}\zeta(s)^{-\rho} = \prod_{p\text{ prime}} (1-p^{-s})^\rho \left(\sum_{k=0}^\infty \frac{B^s}{p^{ks}}\right) \\
        = \prod_{p\text{ prime}} (1-p^{-s})^\rho \frac{1}{1-\frac{B}{p^s}}=\prod_{p \text{ prime}} (1-p^{-s})^\rho \left(1+\frac{B}{p^s-B}\right).
    \end{multline}
    We note that the second term in the product has a pole when $p^s=B$. Thus, on the half-plane $\set{s\in\C: \operatorname{Re}(s)>1-c_0}$ for $c_0\in(0,1/2)$, we need to restrict ourselves to $B\in(0,2^{1-c_0})$.  
    Choosing $\rho=B$, and taking the logarithm of the right-hand side of \eqref{eq:product-alpha-power-of-B}, we get
    \begin{equation*}
        \sum_{p\text{ prime}}B \ln(1-p^{-s}) + \ln\left(1+\frac{B}{p^s-B}\right),
    \end{equation*}
    which is
    \begin{equation*}
        \sum_{p\text{ prime}} - B p^{-s} + B p^{-s} + O_{p\to\infty}(p^{-2s}).
    \end{equation*}
    Checking admissibility as well as admissibility++ is now straightforward.

    \subsection{\texorpdfstring{$\alpha(p^k)=a+O_{p\to\infty}(p^{-\varepsilon k})$ for some $\varepsilon>0$}{alpha(p)=a+O(p pow - epsilon k) for some epsilon>0}}
    We finally consider the case where for some $a>0$ and $\varepsilon>0$, we have
    $\alpha(n)=a+O(n^{-\varepsilon})$, for $n\to\infty$ restricted to the set $n\in\kP$, where 
    \begin{equation*}
    \kP\define\bigcup_{k\in\N} \set{p^k:p\text{ prime}}.
    \end{equation*}
    We assume for simplicity that $\alpha$ is non-negative. Then we note first that
    \begin{equation*}
        \sum_{n\in\N}\frac{\alpha(n)}{n^2} < \infty.
    \end{equation*}
    Second, fix a $c_0\in(0,\sfrac 12)$. Then 
    \begin{equation*}
        \sum_{p\text{ prime}}\left(\sum_{k=1}^\infty\frac{\alpha(p^k)}{p^{k(1-c_0)}}\right)^2 \le \sum_{p\text{ prime}}2\left(\left(\sum_{k=1}^\infty \frac{a}{p^{k(1-c_0)}}\right)^2+\left(\sum_{k=1}^\infty \frac{o_{p\to\infty}(1)}{p^{k(1-c_0)}}\right)^2\right),
    \end{equation*}
    which is indeed finite. 

    Finally, we consider the compensated Euler product for $\rho=a$, 
    \begin{equation*}
        \prod_{p\text{ prime}} \left(1+\sum_{k=1}^\infty\frac{a+O_{p\to\infty}(p^{-\varepsilon k})}{p^{ks}}\right)\left(1-\frac{1}{p^s}\right)^a.
    \end{equation*}
    Taking the logarithm we obtain
    \begin{equation*}
        \sum_{p\text{ prime}} \left(O_{p\to\infty}(p^{-2s})-\frac{a}{p^s}+\sum_{k=1}^\infty\frac{a+O_{p\to\infty}(p^{-\varepsilon k})}{p^{ks}}\right),
    \end{equation*}
    where the first term comes from the quadratic terms of the logarithm, while the second term comes from $(1-p^{-s})^a$ and the last term comes from the first term in the Euler product.
    The sum is convergent whenever $\operatorname{Re}(s)>\max(\frac 12, 1-\varepsilon)$. We have thus established admissibility and, together with the previous reasoning, thus admissibility++.

    \section*{Acknowledgements}
    I thank Pierre-Loic Mél\"{\i}ot and my Ph.D. advisor, Ashkan Nikeghbali, for helpful discussions that formed the foundation of this work. I also extend thanks to Thomas Lehéricy for informative insights leading up to the proof of Proposition~\ref{lem:admissible-function-classes}, and for his multiple helpful remarks on final versions of this text.

    \appendix
    \section{Proof of Proposition~\ref{prop:CLT}}\label{appendix:CLT}
    We assume $\eta''(0)\neq0$. If the cumulant generating function is finite on $\kS_{(c,d)}$, then it is analytic, 
    so in particular it has a Taylor series expansion around $0$. To second order, we have, by Taylor's Theorem,
    \begin{equation}\label{eq:taylor-eta}
        \eta(z) = \eta(0) + \eta'(0) z + \frac{\eta''(0)}2 z^2 + o(z^3)\quad\text{as }z\to0.
    \end{equation}
    We now consider the random variables
    \begin{equation}\label{eq:renormalize}
        Y_n\define\frac{X_n - t_n \eta'(0)}{\sqrt{t_n \eta''(0)}}
    \end{equation}
    Then 
    \begin{equation*}
        \E\left(\e^{z Y_n}\right) = \phi_n\left(\frac{z}{\sqrt{t_n\eta''(0)}}\right) \exp\left(-t_n\eta'(0)\frac{z}{\sqrt{t_n\eta''(0)}}\right).
    \end{equation*}
    For $z$ fixed, we now define the numbers 
    \begin{equation*}
    y_n\define\frac{z}{\sqrt{t_n\eta''(0)}}
    \end{equation*}
    and note that $y_n\to 0$ as $n\to\infty$. Since convergence of $\psi_n$ to $\psi$ is locally uniform, we conclude that 
    \begin{equation*}
        \lim_{n\to\infty}\psi_n(y_n)=\psi(0).
    \end{equation*}
    We also note that, by \eqref{eq:taylor-eta}, and $\eta(0)=0$,
    \begin{equation*}
        \psi_n(y_n)=\phi_n(y_n) \exp\left(-t_n\eta'(0)y_n - t_n \eta''(0) \frac{y_n^2}2 + o(t_n y_n^3) \right).
    \end{equation*}
    Since $\lim_{n\to\infty}t_n y_n^3 = 0$,
    \begin{equation*}
        \lim_{n\to\infty} \exp\left(-t_n\eta''(0)\frac{y_n^2}2 + o(t_n y_n^3)\right) = \exp\left(-\frac{z^2}2\right).
    \end{equation*}
    In conclusion,
    \begin{equation*}
        \lim_{n\to\infty}\E\left(\e^{z Y_n}\right) = \exp\left(\frac{z^2}2\right)\lim_{n\to\infty} \psi_n(y_n) =  \exp\left(\frac{z^2}2\right) \psi(0)
    \end{equation*}
    We now note that
    \begin{equation*}
    \psi(0) = \lim_{n\to\infty} \exp(-t_n\eta(0))\phi_n(0)=1.
    \end{equation*}
    Thus, 
    \begin{equation*}
        \lim_{n\to\infty}\E\left(\e^{z Y_n}\right) = \exp\left(\frac{z^2}2\right).
    \end{equation*}
    Since this holds for all $z\in\mathcal S_{(c,d)}$, it follows from Lévy's continuity Theorem that $Y_n\to Z$ in distribution as $n\to\infty$, where $Z$ is a normally distributed random variable with mean $0$ and variance $1$.\hfill$\square$

\begin{thebibliography}{9}
        \bibitem{BretecheTenenbaum}
        \textsc{Régis de la Bretèche}, \textsc{Gérald Tenenbaum}, \textit{Remarks on the Selberg–Delange method}. 2021, \url{https://arxiv.org/abs/2010.12929}


 
        \bibitem{Elboim-Gorodetsky}
        \textsc{Dor Elboim}, \textsc{Ofir Gorodetsky}, \textit{Multiplicative arithmetic functions and the generalized Ewens measure}. 2022, \url{https://arxiv.org/abs/1909.00601}

        \bibitem{ErdosKac}
        \textsc{Paul Erd\H{o}s}, \textsc{Mark Kac}, \textit{The Gaussian Law of Errors in the Theory of Additive Number Theoretic Functions}. American Journal of Mathematics. 62 (1/4), pages 738–742. 1940.
        
        \bibitem{FMN}
        \textsc{Valentin Féray}, \textsc{Pierre-Loic Mél\"{\i}ot}, \textsc{Ashkan Nikeghbali}, \textit{Mod-$\phi$ convergence I: Normality zones and precise deviations}. 2015, \url{https://arxiv.org/abs/1304.2934}

 

        

        

    

        \bibitem{Tenenbaum}
        \textsc{Gérald Tenenbaum}, \textit{Introduction to Analytic and Probabilistic Number Theory}, third edition, Graduate Studies in Mathematics, Volume 163, American Mathematical Society, 2010.

        \bibitem{Titchmarsh}
        \textsc{Edward Charles Titchmarsh}, \textit{The Theory of Functions}, second edition, Oxford University Press, 1939.
	\end{thebibliography}
\end{document}